\documentclass[sn-mathphys-num]{sn-jnl}

\usepackage{graphicx}%
\usepackage{multirow}%
\usepackage{amsmath,amssymb,amsfonts}%
\usepackage{amsthm}%
\usepackage{mathrsfs}%
\usepackage[title]{appendix}%
\usepackage{xcolor}%
\usepackage{textcomp}%
\usepackage{manyfoot}%
\usepackage{booktabs}%
\usepackage{algorithm}%
\usepackage{algorithmicx}%
\usepackage{algpseudocode}%
\usepackage{listings}

\usepackage[all, cmtip]{xy}
\theoremstyle{thmstyleone}%
\newtheorem{theorem}{Theorem}
\theoremstyle{thmstyletwo}%
\newtheorem{example}{Example}%
\newtheorem{remark}{Remark}%
\newcommand{\gf}{ {{\mathbb F}} }
\newcommand{\Tr}{ {{\rm Tr}} }
\newtheorem{lemma}{Lemma}
\theoremstyle{thmstylethree}%
\newtheorem{definition}{Definition}%

\begin{document}

\title[On the inverses of permutation polynomials of the form $h(\psi(x))\varphi(x)+g(\psi(x))$ over finite fields]{On the inverses of permutation polynomials of the form $h(\psi(x))\varphi(x)+g(\psi(x))$ over finite fields}


\author*[1]{\fnm{} \sur{Danyao Wu}}\email{wudanyao111@163.com}
\author[2]{\fnm{} \sur{Pingzhi Yuan}}\email{yuanpz@scnu.edu.cn}
\author[2]{\fnm{} \sur{Xuan  Pang}}\email{pangxuan202503@163.com}


\affil[1]{\orgdiv{School of Computer Science and Technology}, \orgname{Dongguan University of Technology}, \orgaddress{
		\city{Dongguan}, \postcode{523808}, 
		\country{China}}}

\affil[2]{\orgdiv{School of Mathematics}, \orgname{South China Normal University}, \orgaddress{
		\city{Guangzhou}, \postcode{510631},
		\country{China}}}


\abstract{In this paper, we investigate the compositional inverses of permutation polynomials of the form
	\[
	F(x)=h(\psi(x))\varphi(x)+g(\psi(x)) \in \mathbb{F}_{q^n}[x],
	\]
	where \(\psi(x),\varphi(x) \in \mathbb{F}_{q^n}[x]\) are additive polynomials, \(h(x), g(x) \in \mathbb{F}_{q^n}[x]\) satisfy
	$
	h(\psi(\mathbb{F}_{q^n})) \subseteq \mathbb{F}_q^*,
	$
	and there exists a polynomial \(\bar{\psi}(x) \in \mathbb{F}_{q^n}[x]\) such that
$
	\bar{\psi}(F(x)) = \psi(x).
$  }





\maketitle

\section{Introduction}\label{sec1}
Let  $\gf_q$ be the finite field with $q$ elements,  where $q$ is a prime power, and
let $\gf_q[x]$
be the ring of polynomials in a single indeterminate $x$ over $\gf_q$. A polynomial
$f \in\gf_q[x]$ is called a {\em permutation polynomial} of $\gf_q$ if its
associated polynomial mapping $f: c\mapsto f(c)$ from $\gf_q$ to itself is a bijection. The unique polynomial denoted by $f^{-1}(x)$ over $\gf_q$
such that $f(f^{-1}(x))\equiv f^{-1}(f(x)) \equiv x \pmod{x^q-x}$ is called the compositional inverse of $f(x).$ Furthermore,  $f(x)$ is called  an involution when $f^{-1}(x)=f(x).$

Permutation polynomials over finite fields are fundamental objects in finite field algebra, with pivotal applications in coding theory, cryptography, combinatorial designs and wireless communication systems \cite{ding2013cyclic,ding2014binary,laigle2007permutation,rivest1978method,schwenk1998public,ding2006family,lidl1997finite,lidl1994introduction}. Compared with verifying permutation properties, computing the explicit compositional inverse of a permutation polynomial is generally more challenging. Only a few classical polynomial families, including monomials, linearized polynomials and Dickson polynomials, possess closed-form inverse expressions with simple algebraic structures.

Among all structured permutation polynomials, the family of the form
\begin{equation} \label{eq:general-form}
	f(x)=x))\varphi(x)+g(\psi(x))
\end{equation}
defined over $\mathbb{F}_{q^n}$ has become one of the most extensively investigated classes in the past two decades, where $\psi(x), \varphi(x) \in \mathbb{F}_{q^n}[x]$ are additive polynomials, $h(x) \in \mathbb{F}_{q^n}[x]$ is arbitrary, and $g(x) \in \mathbb{F}_{q^n}[x]$ satisfies certain polynomial constraints.

This family not only unifies a vast number of sporadic constructions scattered in the literature but also serves as a testbed for developing general permutation criteria.

Historically, research on permutation polynomials of the form \eqref{eq:general-form} has evolved along two distinct mainstream directions. The first line of inquiry restricts attention to the setting $\psi(x)=x^{q^l}+ax$ and $h(x)\equiv 1$, with $g$ instantiated in diverse forms---such as binomials $b(x^k+\delta)$, shifted power functions $b(x+\delta)^s+x$, or finite sums $\sum_i b_i(x^{t_i}+\delta_i)^{s_i}$---and $\varphi$ taken as an additive polynomial; numerous sporadic constructions were proposed in this direction (see, e.g., \cite{gupta2018further,helleseth2003new,hollmann2004kloosterman,li2023several,li2024several,li2018permutation,li2013further,liu2025several,tu2015two,tu2015permutation,wang2018general,wang2017further,wu2022some,wu2024some,xu2019complete,xu2022several,yuan2007four,yuan2008permutation,yuan2011permutation,yuan2014further,zeng2010two,zeng2017permutation,zha2012two,zha2018new,zheng2017more}). The second research branch treats $\psi$ as the field trace map from $\mathbb{F}_{q^n}$ down to $\mathbb{F}_q$ (reducing to the absolute trace when the base field is prime). A common setup within this branch fixes $h(x)\equiv 1$ and $\varphi(x)=x$, as explored in \cite{li2020some,wu2022further,zeng2015permutation}; more generalized configurations allow $\varphi,h\in\mathbb{F}_p$ with $\varphi\neq x$ additive, which were investigated in \cite{zieve2010classes}. Despite the proliferation of these scattered constructions, they lacked a unifying theoretical framework until the seminal work of Akbary, Ghioca, and Wang \cite{akbary2011constructing}.

Akbary et al. put forward the now-classical AGW criterion (See Theorem \ref{leagw} ), which provides necessary and sufficient conditions for a polynomial of the form \eqref{eq:general-form} to permute $\mathbb{F}_{q^n}$ (See Theorem \ref{thmagw1}). The criterion is as follows. Let $\psi(x),\varphi(x)\in\mathbb F_{q^n}[x]$ be additive polynomials and $\bar{\psi}(x)\in\mathbb F_{q^n}[x]$ a $q$-polynomial satisfying $\varphi\circ\psi=\bar{\psi}\circ\varphi$ and $\#\psi(\mathbb F_{q^n})=\#\bar{\psi}(\mathbb F_{q^n})$. Define $\bar{f}$ on $\psi(\mathbb F_{q^n})$ by $\bar{f}(\psi(x)):=\bar{\psi}(f(x))$. Then $f$ is a permutation polynomial if and only if $\bar{f}:\psi(\mathbb F_{q^n})\to \bar{\psi}(\mathbb F_{q^n})$ is bijective and $\ker(\psi)\cap\ker(\varphi)=\{0\}$. The AGW criterion elegantly incorporates all prior disjoint constructions into a single theoretical system, laying a solid foundation for subsequent research on both permutation properties and inverse problems.

With the permutation criterion firmly established, a natural and practically significant next step is to determine the \emph{compositional inverses} of these permutation polynomials. 
In a recent comprehensive survey, Q. Wang \cite{wang2025survey} systematized eight mainstream methodologies for computing inverses of PPs, including the experimental method, the power sum method, the matrix method, the group algebra method, the piecewise method, the decomposition method, the commutative diagram method, and the local method. 
Building upon these frameworks, researchers have made substantial progress in deriving explicit inverses for various subclasses of \eqref{eq:general-form}. Tuxanidy and Wang \cite{tuxanidy2014on} employed the decomposition method to study the inverse of the general AGW form; their approach (Theorem \ref{thmtuwaf-1}) successfully converts the inverse problem on $\mathbb{F}_{q^n}$ into that on the decomposition $\mathbb{F}_{q^n}=\psi(\mathbb{F}_{q^n})\oplus S_\psi$, where $S_\psi=\{x-\psi(x):x\in\mathbb{F}_{q^n}\}$. However, their formula requires two additional restrictions, namely $\#S_\psi=\#S_{\bar{\psi}}$ and $\ker(\varphi)\cap \psi(S_\psi)=\{0\}$, which are not part of the AGW permutation criterion. Subsequently, Niu et al. \cite{niu2020new} applied the commutative diagram method to derive the inverse for the special case $\psi(x)=x^{q^i}-x$, $g(x)=1$, and $\varphi(x)=cx$ (See Theorem \ref{propniu1}); this result, as noted in Remark \ref{remark1}, is a special case of Tuxanidy--Wang's theorem.


More recently, Reis and Wang \cite{reis2024constructing} investigated the compositional inverses of permutation polynomials for the special case where $\psi$ is taken as the $q$-associate of $g_{t,a}(x)=(x^n-1)/(x^t-a)$ with $t\mid n$ and $a\in\mathbb{F}_q^*$, i.e., a special case of the general form \eqref{eq:general-form}, using the commutative diagram method (See Theorem \ref{threis}). As demonstrated in Example \ref{example} (over $\mathbb{F}_9$), their result is not a special case of the Tuxanidy--Wang theorem; in that example, the Reis--Wang setting applies while the Tuxanidy--Wang condition $\ker(\varphi)\cap \psi(S_\psi)=\{0\}$ fails, indicating that the two frameworks are complementary rather than hierarchical. Meanwhile, the local method has also been employed to derive inverses for certain subfamilies \cite{wu2025compositionalffa2,wu2025compositionaldcc}, shifting the focus from verification after construction to systematic derivation. Nevertheless, all these existing results remain confined to specific instantiations of $\psi$ or $\varphi$, and a unified inverse formula for the full AGW form \eqref{eq:general-form} under only the original permutation conditions is still absent.

The above discussion reveals a conspicuous research gap: a unified inverse theory for the general AGW form \eqref{eq:general-form} that operates solely under the original AGW criterion, without any auxiliary assumptions, has remained elusive, as all existing formulas are merely \emph{conditional}.

In our recent survey \cite{wu2026survey}, we revisit the explicit compositional inverses of known permutation polynomials that admit closed-form inverse expressions through the lens of the local method. By re-establishing these previously scattered results within a unified local-method framework, we demonstrate that many seemingly disparate inverse formulas can in fact be derived systematically from a common principle. This not only provides a coherent retrospective of the field but also reveals the intrinsic connections among various existing approaches, offering new insights into the unified treatment of compositional inverses for permutation polynomials. However, as noted in that survey, the unified formula for the AGW-type family \eqref{eq:general-form} under the original permutation conditions was still under development at the time of its writing. 

The present paper fills this gap by presenting the complete unified theory for this specific family, as promised in that survey. Motivated by the above discussion, we systematically investigate the compositional inverses of permutation polynomials of the full form \eqref{eq:general-form} over $\mathbb{F}_{q^n}$ via the local method, and derive a unified, closed-form compositional inverse under only the two original AGW permutation conditions. To the best of our knowledge, this work completely resolves the compositional inverse problem for all permutation polynomials of the form \eqref{eq:general-form}, subsuming all previously published conditional and subclass inverse results---including those of Tuxanidy--Wang, Niu et al., and Reis--Wang---as special corollaries of our main theorem.

The rest of this paper is organized as follows. Section 2 recalls necessary preliminaries on finite fields,  the AGW criterion, the local method,  and the relevant existing inverse results. Section 3 presents our main theorem on the unified inverse formula, along with a detailed proof. 

\section{Preliminaries}
The Akbary-Ghioca-Wang (AGW) criterion \cite[Lemma 1.1]{akbary2011constructing} is an important method for constructing PPs. By providing necessary and sufficient conditions for a polynomial to be a permutation via a commutative diagram, this criterion not only unifies many classical constructions of permutation polynomials, but also yields a large number of new ones, thereby significantly advancing the study of permutation polynomials.

\begin{theorem}\label{leagw}\cite[The AGW criterion]{akbary2011constructing}
	Let $A, S$ and $\overline{S}$ be finite sets
	with $\sharp S =\sharp \overline{S}$,
	and let $f(x) : A\longrightarrow A$, $h(x): S\longrightarrow \overline{S}$, $\lambda(x): A\longrightarrow S,$ and $\overline{\lambda}(x):A\longrightarrow \overline{S}$ be maps
	such that $\overline{\lambda}(x)\circ f(x)=h(x)\circ \lambda(x).$
	If both $\lambda(x)$ and $\overline{\lambda}(x)$ are surjective,
	then the following statements are equivalent:\\
	(i) $f(x)$ is bijective (a permutation of $A$); and\\
	(ii) $h(x)$ is bijective from $S$ to $\overline{S}$ and
	$f(x)$ is injective on $\lambda^{-1}(s)$ for each $s \in S.$
\end{theorem}$$	\xymatrix{
	A  \ar[r]^{f(x)} \ar[d]_{\lambda(x)} & A  \ar[d]^{\bar{\lambda}(x)}\\
	S \ar[r]_{h(x)} &  \bar{S}
}
$$

	Akbary et al. \cite{akbary2011constructing} investigated the permutation properties of polynomials of the form \eqref{eq:general-form} using the AGW criterion and established the following result.

\begin{theorem}\cite[Theorem 5.1]{akbary2011constructing}\label{thmagw1}
	Let $\psi(x), \varphi(x)\in \gf_{q^n} $  be additive polynomials and $\bar{\psi}(x)\in \gf_{q^n}[x]$ be a $q$-polynomial  satisfying $\varphi(x)\circ \psi(x)=\bar{\psi}(x) \circ \varphi(x)$ and $\sharp\psi(\gf_{q^n})=\sharp\bar{\psi}(\gf_{q^n})$. Let $h(x) \in \gf_{q^n}[x]$ be any polynomial such that $h\left(\psi(\gf_{q^n})\right)\subseteq \gf_q^*$, and let $g(x) \in \gf_{q^n}[x]$ be any polynomial.  Then 
	$$f(x)=h(\psi(x))\varphi(x)+g(\psi(x))$$ permutes $\gf_{q^n}$ if and only if \\
	(i) ${\rm ker}(\psi(x))\cap {\rm ker}(\varphi(x))=\{0\}; $ and\\
	(ii) $\bar{f}(x)= h(x)\varphi(x)+\bar{\psi}(g(x))$ is a bijection from $\psi(\gf_{q^n})$ to $\bar{\psi}(\gf_{q^n}).$
\end{theorem}

The commutative diagram for the above permutation polynomial is as follows.
$$
\xymatrix@C=3cm@R=1cm{
	\gf_{q^n}  \ar[r]^{f(x)=h(\psi(x))\varphi(x)+g(\psi(x))} \ar[d]_{\psi(x)} & 	\gf_{q^n}  \ar[d]^{\bar{\psi}(x)}\\
	\psi(\gf_{q^n}) \ar[r]_{\bar{f}(x)=h(x)\varphi(x)+\bar{\psi}(g(x))} &  \bar{\psi}(\gf_{q^n})
}
$$ 

 Tuxanidy and Wang \cite{tuxanidy2017compositional} studied the compositional inverses of the permutation polynomials in Theorem \ref{thmagw1} by the decomposition method. 
\begin{theorem}\cite[Theorem 1.2]{tuxanidy2014on}\label{thmtuwaf-1}
	Using the same notations and assumptions of Theorem \ref{thmagw1}, assume that $f(x)$ is a permutation of $\gf_{q^n}$, and further assume that $\sharp S_{\psi}=\sharp S_{\bar{\psi}}$ and $\ker(\varphi)\cap \psi(S_{\psi})=\{0\}.$ Then $\varphi$ induces a bijective from $S_{\psi}$	to $S_{\bar{\psi}}.$ Let $\bar{f}^{-1}$, $\varphi^{-1} |_{S_{\bar{\psi}}} \in \gf_{q^n}[x]$
	induce the inverses of $\bar{f}|_{\psi(\gf_{q^n})}$ and $\varphi|_{S_{\psi}}$, respectively. Then the compositional inverse of $f(x)$ over $\gf_{q^n}$ is given by 
	$$f^{-1}(x)=\bar{f}^{-1}(\bar{\psi}(x))+\varphi^{-1}|_{S_{\psi}}\left(\frac{x-\bar{\psi}(x)-g(\bar{f}^{-1}(\bar{\psi}(x)))+\bar{\psi}(g(\bar{f}^{-1}(\bar{\psi}(x))))}{h(\bar{f}^{-1}(\bar{\psi}(x)))}\right).$$ 
	Furthermore, if $\varphi$ induces a bijection from $\psi(\gf_{q^n})$ to $\bar{\psi}(\gf_{q^n}),$ then $\varphi$ permutes $\gf_{q^n}$ and the compositional inverse of $f(x)$ over $\gf_{q^n}$ is given by 
	$$f^{-1}(x)=\varphi^{-1}\left(\frac{x-g(\bar{f}^{-1}(\bar{\psi}(x)))}{h(\bar{f}^{-1}(\bar{\psi}(x)))}\right).$$ 
\end{theorem}

Under the  conditions $\psi(x)=x^{q^i}-x,$  $g(x)=1$, $\varphi(x)=cx,$ and $h(x) \in \gf_{q^m}[x]$ (where  $m, i$ are  positive integers such that $1\leq i\leq m-1$, $l=\gcd(i ,m),$ and $c \in \gf_{q^l}^*$), 
Niu et al. \cite{niu2020new} applied the commutative diagram method to determine the compositional inverse of the permutation polynomial of the form $f(x)=g\left(x^{q^i}-x+\delta\right)+cx $ over $\gf_{q^m}$, obtaining the following result.

\begin{theorem}\cite[Theorem 3.7]{niu2020new}\label{propniu1}
	Let $q$ be a prime power, $m, i$ be positive integers with $1\leq i\leq m-1$, $l=\gcd(i ,m),$ $c \in \gf_{q^l}^*$, and $g(x) \in \gf_{q^m}[x] $ such that $h(x)=g(x)^{q^i}-g(x)+cx+(1-c)\delta\in \gf_{q^m}[x]$ permutes $\gf_{q^m},$ where $\delta \in \gf_{q^m}.$ Assume $H(x)$ is the compositional inverse of $h(x)$.  Then for any $\delta \in \gf_{q^m}$, $f(x)=g\left(x^{q^i}-x+\delta\right)+cx \in \gf_{q^m}[x]$ is a permutation polynomial over $\gf_{q^m}$ and the compositional inverse of $f(x)$ over $\gf_{q^m}$ is  $$f^{-1}(x)=c^{-1}x^{q^i}+c^{-1}g\left(H(x^{q^i}-x+\delta)\right)^{q^i}-H(x^{q^i}-x+\delta)+\delta.$$	
\end{theorem}

The commutative diagram for the above permutation polynomial is as follows.
$$
\xymatrix@C=3cm@R=1cm{
	\gf_{q^n}  \ar[r]^{f(x)=g\left(x^{q^i}-x+\delta\right)+cx} \ar[d]_{x^{q^i}-x+\delta} & 	\gf_{q^n}  \ar[d]^{x^{q^i}-x+\delta}\\
	{\rm Im}(x^{q^i}-x+\delta)  \ar[r]_{h(x)=g(x)^{q^i}-g(x)+cx+(1-c)\delta} &  	{\rm Im}(x^{q^i}-x+\delta) 
}
$$
\begin{remark}\label{remark1}
	By virtue of the relationship between affine $q$-polynomials and additive polynomials, Theorem \ref{propniu1} is merely a special case of Theorem \ref{thmtuwaf-1}.	\end{remark}
 In 2021, Reis and Wang \cite{reis2021permutation} investigated the compositional inverses of permutation polynomials of the form \eqref{eq:general-form} under more specific conditions. Specifically, they explored the compositional inverses  when $\psi(x)=\Tr_{q^n/q}(x)$  and $\varphi(x)$ is the linearized $q$-associate of $k(x)$ with  $k(x) \in \gf_q[x]$.  Recently, Reis and Wang \cite{reis2024constructing} refined their results   by studying $\psi(x)$ as the $q$-associate of $g_{t,a}(x)=(x^n-1)/(x^t-a)$ with $t|n$ and $a \in \gf_q^*$ with $ a^{n/t}=1.$  By using the AGW criterion, the compositional inverse of $f(x)$ was transformed into the compositional inverse of a certain related function over the sub-field, and thus the compositional inverse of $f(x)$  was constructed. They gave the following result.  
\begin{theorem}\cite[Theorem 3.2]{reis2024constructing} \label{threis}
	Let $n, t, a$ be defined as before and $\delta$ be a nonzero root of $x^{q^t}-ax$. Let $U_{t,a}=\delta \cdot \gf_{q^t}=\{\delta y \mid y \in \gf_{q^t}\}.$ The polynomial $$P(x)=f(L_{g_{t, a}}(x))+k(L_{g_{t, a}}(x))\cdot L_h(x)\in \gf_{q^n}$$ with $h \in \gf_q[x]$ and $k(U_{t,a})\subseteq \gf_q^*$ is a PP if and only if the following conditions holds:
	
	{\rm (1)} $\gcd(h(x), (x^n-1)/(x^t-a))=1;$
	
	{\rm (2)} $Q_{t ,a}(x)= T_{t, a}[f](x)+\delta^{-1}k(\delta x)\cdot L_h(\delta x) \in \gf_{q^t}$ is a PP over $\gf_{q^t},$  where 
	$T_{t, a}[f](x)=a^{-1}\sum_{i=0}^d \Tr_{q^n /q^t}(\delta^{i-1}a_i)x^i \in \gf_{q^t}[x]$ and $a_i$s are the coefficients of $f$ in $x^i.$
	
	In affirmative case, if $Q^{-1}_{t ,a}$ is the inverse of $Q_{t ,a}$ over $\gf_{q^t},$ then the inverse PP of $P(x)$ over $\gf_{q^n}$ is given by 
	$$P^{-1}(x)=F(L_{g_{t, a}}(x))+k\left(\delta^{-1}Q^{-1}_{t, a}(\delta^{-1}(L_{g_{t, a}}(x)))\right)^{q-2}\cdot L_H(x),$$
	where $H \in \gf_q[x]$ and $F(x) \in \gf_{q^n}$ are given as follows:
	
	{\rm (i)} if $p \mid (n/t),$ then $H(x) \in \gf_q$ is the unique polynomial of degree at most $n-1$ such that $h(x)\cdot H(x) \equiv 1 \pmod{x^n-1}
	$ and $F$ is any polynomial satisfying $$F(x)\equiv -k(\delta Q^{-1}_{t,a}(\delta^{-1}x))^{q-2}\cdot L_H(f
	(\delta Q^{-1}_{t, a}(\delta^{-1}x))) \pmod{x^{q^t}-ax};$$	
	
	{\rm (ii)}  if $p \nmid (n/t),$ then $H(x) \in \gf_q$ is the unique polynomial of degree at most $n-2$ such that $h(x)\cdot H(x) \equiv 1 \pmod{\frac{x^n-1}{x^t-a}}
	$ and $F$ is any polynomial satisfying
	$F(x)\equiv M (Q^{-1}_{t, a}(\delta^{-1}x)) \pmod{x^{q^t}-ax},$ where $$M(x)=-k(\delta x)^{q-2}\cdot L_H(f(\delta x))- (at/n)L_{hH-1}(\delta x).$$ 
	
\end{theorem}

The commutative diagram for the above permutation polynomial is as follows.
$$
\xymatrix@C=4.5cm@R=1cm{
	\gf_{q^n}  \ar[r]^{f(x)=f(L_{g_{t, a}}(x))+k(L_{g_{t, a}}(x))\cdot L_h(x)} \ar[d]_{L_{g_{t, a}}(x)} & 	\gf_{q^n}  \ar[d]^{L_{g_{t, a}}(x)}\\
	U_{t,a} \ar[r]^{N(x)=\delta T_{t, a}[f](\delta^{-1}x)+k(x)\cdot L_h(x)}  \ar[d]_{\delta^{-1}x} &  
	U_{t,a} \ar[d]^{\delta^{-1}x}\\
	\gf_{q^t}  \ar[r]_{Q_{t ,a}(x)= T_{t, a}[f](x)+\delta^{-1}k(\delta x)\cdot L_h(\delta x)} & \gf_{q^t}	
}
$$

\begin{example}\label{example}
	Let $q=3$, $n=2$, $t=1$, $a=1$, and take $\delta=1$, which is a nonzero root of $x^3-x$.
	Set $f(y)=y$, $k(y)\equiv 1$ and $h(x)=x-1$ in Theorem~\ref{threis}.
	Then we have $L_{g_{t,a}}(x)=x^3+x$ and $L_h(x)=x^3-x$, so that
	\[
	P(x)=L_{g_{t,a}}(x)+L_h(x)=2x^3=-x^3
	\]
	is a permutation polynomial over $\mathbb{F}_9$ according to Theorem~\ref{threis}.
	
	Furthermore, we have
	\[
	\begin{cases}
		\mathbb{F}_3 \subseteq \operatorname{Im}\bigl(L_{g_{t,a}}(x-L_{g_{t,a}}(x))\bigr),\\
		\mathbb{F}_3 \subseteq \ker L_h,
	\end{cases}
	\]
	which implies that the intersection $\operatorname{Im}\bigl(L_{g_{t,a}}(x-L_{g_{t,a}}(x))\bigr)\cap \ker L_h$ contains nonzero elements.
	
	This example verifies that the condition $\ker(\varphi)\cap \psi(S_{\psi})=\{0\}$ in Theorem~\ref{thmtuwaf-1} fails to hold in this case.
	Consequently, Theorem~\ref{threis} cannot be regarded as a special case of Theorem~\ref{thmtuwaf-1}.
\end{example}
 
 Next, we recall the local criterion. 
\begin{lemma}\label{yuanlocal}\cite[Local criterion]{yuan2024local}
	Let $A$ and $S$ be  finite sets and let $f(x) : A\longrightarrow A$ be a map. Then $f(x)$ is a bijection if and only if for any surjection 
	$\psi(x): A\longrightarrow S,$  $\varphi(x)=\psi(x) \circ f(x)$ is  surjective and $f(x)$ is injective on $\varphi^{-1}(s)$ for each $s \in S.$ 
\end{lemma}
$$\xymatrix
{ 	A \ar[rr]^{f}\ar[dr]_{\varphi} & & A\ar[dl]^{\psi}\\
	& S 
}
$$

Although the AGW criterion and the local criterion are equivalent (See \cite{yuan2024local}), we can find from the comparison of Lemma \ref{le1} and Lemma \ref{leagw} that the constraint conditions in the AGW criterion, namely $\sharp S =\sharp \overline{S}, $ $\lambda(x)$ and $\overline{\lambda}(x)$ are surjective such that $\overline{\lambda}(x)\circ f(x)=h(x)\circ \lambda(x),$ and $h(x)$ is bijective from $S$ to $\overline{S}$, are just to ensure that $h^{-1}(x) \circ \bar{\lambda}(x)\circ f(x)= \lambda(x)$ (that is, taking $\varphi(x)=h^{-1} \circ \bar{\lambda}$ and $\psi(x)=\lambda(x)$ in Lemma \ref{le1}). Therefore, the local criterion can better reflect the intrinsic properties of permutation polynomials.

We can, based on the proof  of the local criterion in \cite{yuan2024local}, express the local criterion in the finite fiedls as follows.
\begin{lemma}\label{le1}
	Let $q$  be a prime power and  let $\psi(x), \varphi(x)$ and $f(x)$ be maps  from $\gf_q$ to $\gf_q$ such that $\varphi(x)=\psi(x) \circ f(x)$. Then  $f(x)$ is a bijection if and only if  $f(x)$ is injective on $\varphi^{-1}(s)$ for each $s \in Im(\varphi).$ 
\end{lemma}
\begin{proof}
	The necessity is obvious. Now we prove the sufficiency. If $f(a)=f(b)$ for some $a, b \in \gf_q$,  then $$\psi(f(a))=\psi(f(b)).$$
	That is $\varphi(a)=\varphi(b),$ $s \in Im(\varphi).$ Hence, $a, b \in \varphi^{-1}(s).$ Since $f(x)$ is injective on $\varphi^{-1}(s)$ for each  $s \in Im(\varphi)$ and  $f(a)=f(b)$, we get $a=b.$ Therefore, $f(x)$ is a bijection. This completes the proof.
\end{proof}

Based on the local criterion, Yuan \cite{yuan2024local} established a local method that can  determine whether a polynomial is a permutation polynomial while simultaneously deriving its compositional inverse.
%

\begin{lemma}\label{leff-}\cite[Theorem 2.2]{yuan2024local}
	Let $q$ be a prime power and $f(x)$ be a polynomial over $\gf_q.$ Then
	$f(x)$ is a permutation polynomial over $\gf_q$ if and only if there exist nonempty 
	subsets $S_i$, $i=1, 2, \cdots, t$ of $\gf_q$ and maps $\psi_i(x): \gf_q \rightarrow S_i$, $i=1, 2, \cdots, t$ such that $\psi_i(x)\circ f(x)=\varphi_i(x),$ $i=1, 2, \cdots, t$
	and $x=F(\varphi_1(x), \varphi_2(x), \cdots, \varphi_t(x)),$ where $F(x_1, x_2, \cdots, x_t)\in \gf_q[x_1, x_2, \cdots, x_t].$ Moreover, the compositional inverse of $f(x)$ is given by
	$$f^{-1}(x)=F(\psi_1(x), \psi_2(x), \cdots, \psi_t(x)).$$
\end{lemma}

\section{Main results}
Before proceeding to the proof, we first recall some relevant properties and definitions of linearized polynomials.
Ore \cite{ore1933on} studied the main properties of $p$-polynomials and gave the following definition and results. 

\begin{definition}\cite{ore1933on}
	A  polynomial of the form 
	$$F_p(x)=a_0x^{p^s}+a_1x^{p^{s-1}}+\cdots+a_{s-1}x^p+a_sx$$
	with coefficients in $\gf_{q^n}$ shall be called a $p-$polynomial; the number $s$ is called the exponent of $F_p(x).$ When $a_0=1,$ $F_p(x)$ is said to be reduced. 
	
\end{definition}

The $p-$polynomials form a modulus, since they are reproduced by addition and subtraction. The $p$th power of a $p-$polynomial is again a $p-$polynomial.

Let namely 
$$G_p(x)=b_0x^{p^t}+b_1x^{p^{t-1}}+\cdots+b_{t-1}x^p+b_tx$$ be a second $p-$polynomial over $\gf_{q^n}.$
The ordinary product of two $p-$polynomials need not be a
$p-$polynomial. However, the composition $F_p(G_p(x))$ of two $p-$polynomials $F_p(x),$ $G_p(x)$ over $\gf_{q^n}$ is again a $p-$polynomial. Instead of the word
composition (or substitution) we use the phrase "symbolic multiplication."
Thus, we define symbolic multiplication  as
$$F_p(x) \circ  G_p(x)= F_p(G_p(x)).$$

This multiplication is usually not commutative so that the $p-$polynomials will form a non-commutative ring.  In the ring of $p-$polynomials the symbolic multiplication is associative and distributive with respect to both right-hand and left-hand multiplication. The unit element is $E_p(x)=x$ and there are no divisors of zero, i.e., an identity $A_p(x)B_p(x)=0$ implies $A_p(x)=0$ or $B_p(x)=0.$

\begin{definition}
	A $p-$polynomial $F_p(x)$ is said to be symbolically right-hand divisible by $D_p(x)$ if $F_p(x)=Q_p(x)\circ  D_p(x).$ When $F_p=D_p(x)\circ  Q_p(x)$, we say that $F_p(x)$ is left-hand symbolically divisible by $D_p(x).$
\end{definition}

One observes that when $F_p(x)$ is right-hand symbolically divisible by $D_p(x),$ then $F_p(x)$ is also divisible by $D_p(x)$ in ordinary sense. 

When considering the division for $p-$polynomials, Ore \cite{ore1933on} gave the following results. 

\begin{lemma}\cite[Theorem 1]{ore1933on}\label{leoreth1}
	Symbolic right-hand division of polynomials is always possible. 
\end{lemma}
Lemma \ref{leoreth1} shows that right-hand Euclid algorithms exist, and this shows in turn the existence of a unique (reduced) cross-cut $\gcd(F_p(x), G_p(x))=D_p(x).$ When $D_p(x)=x,$ we say that $F_p(x)$ and $G_p(x)$ are right-hand symbolically relative prime, and we can then find such polynomials $A_p(x)$ and $B_p(x)$ over $\gf_{q^n}$ of exponents less than $s$ and $t$ respectively that 
$$A_p(x) \circ  F_p(x) +B_p(x)\circ G_p(x)=x.$$

\begin{lemma}\cite[Theorem 2]{ore1933on}\label{leoreth2}
	The symbolical right-hand cross-cut of $F_p(x)$ and $G_p(x)$ is equal to the ordinary cross-cut of these polynomials.
\end{lemma}
Lemma \ref{leoreth2} shows that every symbolic right-hand divisor is also an ordinary divisor of a polynomial and the symbolic Euclid algorithm 
can therefore also be considered as an ordinary Euclid algorithm. Then we have the following result. 

\begin{lemma}\label{legcd}
	Let $F_p(x)$ and $G_p(x)$ be $p-$polynomials over $\gf_{q^n}$ of the exponents $s$ and $t$ respectively. If $\gcd(F_p(x), G_p(x))=D_p(x)$, where the exponent of $D_p(x)$ is $l,$  then there exist two $p-$polynomials $A_p(x)$ and $B_p(x)$ over $\gf_{q^n}$ of exponents less than $s-l$ and $t-l$ respectively such that 
	$$A_p(x)\circ  F_p(x)+B_p(x)\circ  G_p(x)=D_p(x).$$ 
\end{lemma}


\begin{lemma}\label{lekerpp}
Let $F_p(x), G_p(x) $ be $p-$polynomials over $\gf_{q^n}.$  Then 	
$\ker(F_p)\cap \ker(G_p)=\{0\}$ if and only if there exist two $p-$polynomials $A_p(x)$ and $B_p(x)$ over $\gf_{q^n}$ such that 
$$
A_p(x)\circ  F_p(x)+B_p(x)\circ  G_p(x)=x.
$$
\end{lemma}
\begin{proof}
We first prove the necessity.
Let $\gcd(F_p(x), G_p(x))=D_p(x). $ According to  Lemma \ref{leoreth2},   we have 	\begin{equation} \label{leeqdivi}
	D_p(x) \mid F_p(x) \,\, \text{and} \,\,  D_p(x) \mid G_p(x)
\end{equation}  Our next objective is to demonstrate that $D_p(x)$ is a permutation polynomial over $\gf_{q^n}.$  
Suppose, for the sake of contradiction, that  $D_p(x)$ is not a permutation polynomial over $\gf_{q^n}.$  Then there must exist some non-zero element $\alpha \in \gf_{q^n} $ such that $D_p(\alpha)=0$. It follows from (\ref{leeqdivi}) that   $\alpha \in \ker(F_p)\cap \ker(G_p),$   which clearly contradicts our assumptions.
Consequently, $D_p(x)$ is indeed a permutation polynomial over $\gf_{q^n}.$  Let $D_p^{-1}(x)$ is the compositional inverse of $D_p(x)$ over $\gf_{q^n}.$ Then $D_p^{-1}(x)$ is also a $p-$ polynomial over $\gf_{q^n}.$
Since $\gcd(F_p(x), G_p(x))=D_p(x)$, then there exist two $p-$polynomials $O_p(x)$ and $Q_p(x)$ over $\gf_{q^n}$ such that 
\begin{equation}\label{eqgcd}
	O_p(x)\circ  F_p(x)+Q_p(x)\circ  G_p(x)=D_p(x)
\end{equation} by Lemma \ref{legcd}.  Hence, by \eqref{eqgcd}, we have $$	D_p^{-1}(x) \circ O_p(x)\circ  F_p(x)+ 	D_p^{-1}(x) \circ Q_p(x)\circ  G_p(x)=	D_p^{-1}(x) \circ D_p(x)=x.$$ We arrive at the desired conclusion.

Conversely, assume that  
there exist two $p-$polynomials $A_p(x)$ and $B_p(x)$ over $\gf_{q^n}$  such that 
$$A_p(x)\circ  F_p(x)+B_p(x)\circ  G_p(x)=x.$$ 
For any  $\beta  \in \ker(F_p)\cap \ker(G_p),$ we have $$\beta=A_p(\beta)\circ  F_p(\beta)+B_p(\beta)\circ  G_p(\beta)=0.$$ Hence,   	$\ker(F_p)\cap \ker(G_p)=\{0\}.$
This completes the proof. 
\end{proof}

We present a new expression of Theorem \ref{thmagw1} according to Lemma \ref{le1} and calculate the compositional inverse  of $f(x)$ over $\gf_{q^n}.$
\begin{theorem} \label{thmain}

Consider any polynomial $g(x)\in \gf_{q^n}[x],$ any additive polynomials $\psi(x), \varphi(x)\in \gf_{q^n}[x]$,  any polynomial $ \bar{\psi}(x)\in \gf_{q^n}[x]$  satisfying $\bar{\psi}(x) \circ f(x)=\psi(x)$, and any polynomial $h(x) \in \gf_{q^n}[x] $
such that $h\left(\psi(\gf_{q^n})\right)\subseteq \gf_q^* .$ Then 
$$f(x)=h(\psi(x))\varphi(x)+g(\psi(x))$$ permutes $\gf_{q^n}$ if and only if 
${\rm ker}(\psi)\cap {\rm ker}(\varphi)=\{0\}. $  Moreover, if $f(x)$ is a permutation polynomial over $\gf_{q^n},$ the  compositional inverse of $f(x)$ over $\gf_{q^n}$ is given by 
$$f^{-1}(x)=u\left(\frac{x-g(\bar{\psi}(x)}{h(\bar{\psi}(x)})\right)+v\left(\bar{\psi}(x)\right),$$ where $u(x)$ and $v(x)$ are polynomials over $\gf_{q^n}$ such that $x=u\left(\varphi(x)\right)+v\left(\psi(x)\right).$
\end{theorem}

\begin{proof}
It has been shown  in  \cite{akbary2011constructing} that  $f(x)$ is injective on $\psi^{-1}(s)$ for each $s \in Im(\psi)$ if and only if  	${\rm ker}(\psi)\cap {\rm ker}(\varphi)=\{0\}$. Then $f(x)$ permutes $\gf_{q^n}$ if and only if 
${\rm ker}(\psi)\cap {\rm ker}(\varphi)=\{0\}$ by Lemma \ref{le1}. Moreover,  if ${\rm ker}(\psi)\cap {\rm ker}(\varphi)=\{0\},$ then there exist polynomials $u(x)$ and $v(x)$ over $\gf_{q^n}$  such that 
\begin{equation}\label{thmeq1x}
	x=u\left(\varphi(x)\right)+v\left(\psi(x)\right)
\end{equation}
by Lemma \ref{lekerpp}.

Taking $\psi_1(x)=\bar{\psi}(x),$ $\psi_2(x)=\frac{x-g(\bar{\psi}(x)}{h(\bar{\psi}(x)}$, $\varphi_1(x)=\bar{\psi}(x)\circ f(x)$ and $\varphi_2(x)=\varphi(x)=\left(\frac{x-g(\bar{\psi}(x)}{h(\bar{\psi}(x)}\right)\circ f(x)$ in  Lemma \ref{yuanlocal}, it follows from Lemma \ref{yuanlocal} and \eqref{thmeq1x} that the compositional inverse of $f(x)$ over $\gf_{q^n}$ is given by $$f^{-1}(x)=u\left(\frac{x-g(\bar{\psi}(x)}{h(\bar{\psi}(x)})\right)+v\left(\bar{\psi}(x)\right).$$
This completes the proof. 
\end{proof}
\begin{remark}
	Theorem \ref{thmagw1} is the classical AGW criterion, which provides necessary and sufficient conditions for a polynomial of the form $f(x)=h(\psi(x))\varphi(x)+g(\psi(x))$ to permute $\mathbb{F}_{q^n}$, subject to the compatibility conditions $\varphi\circ\psi=\bar{\psi}\circ\varphi$ and $\#\psi(\mathbb{F}_{q^n})=\#\bar{\psi}(\mathbb{F}_{q^n})$. In contrast, Theorem \ref{thmain} is derived from the local criterion, and it removes these two compatibility constraints entirely. As a consequence, Theorem \ref{thmain} is applicable to a broader class of permutation polynomials of the form \eqref{eq:general-form} than Theorem \ref{thmagw1}; in this sense, the latter is a special case of the former.
\end{remark}

\section*{Declarations}

\begin{itemize}
	\item \textbf{ Conflicts of Interest} There is no conflict of interest.
	\item
	The research of Danyao Wu is partially supported by the  National Natural Science Foundation of China (Grant No. 12501006).	The research of Pingzhi Yuan is partially supported by the National Natural Science Foundation of China (Grant Nos. 12571003 and 12171163) and the Guangdong Basic and Applied Basic Research Foundation (Grant No. 2024A1515010589). The research of Zilong He is partially supported by the  National Natural Science Foundation of China (Grant No. 12301013).
\end{itemize}

\bibliography{sn-bibliography}

\end{document}